\documentclass[12pt,reqno]{amsart}

\usepackage[mathlines]{lineno}

\usepackage{times}
\usepackage{amsfonts}
\usepackage{amssymb}
\usepackage{latexsym}
\usepackage{xcolor}

\newtheorem{theorem}{Theorem}
\newtheorem{lemma}[theorem]{Lemma}
\newtheorem{prop}[theorem]{Proposition}
\newtheorem{definition}[theorem]{Definition}
\newtheorem{remark}{Remark}

\newtheorem{example}[theorem]{Example}

\usepackage{comment}

\begin{document}
\begin{center}
ON LAPLACE TRANSFORM ON  SEMILATTICES
\end{center}

\

\begin{center}
A. R. Mirotin
\end{center}

\

\begin{flushright}
\textit{Dedicated to the memory of Yu. I. Lyubich }
\end{flushright}

\

\section*{Abstract} The aim of this work is  to prove inverse formulas for Laplace transform on semilattices of open-and-compact sets in a both discrete and non-discrete cases. These are partial answers to a question posed by  Yu.~I.~Lyubich.

\medskip

{\bf 2020 MSC:} Primary 22A20, 22A26

{\bf Keywords:} Idempotent semigroup, semicaracter, Laplace transform, inverse formula, semilattice, topological semilattice.
\bigskip

\

\section{Introduction}

Recall that a semicharacter of a semigroup  $S$ with unit $e$ is any bounded complex function  on $S$ that satisfies the relations
\begin{equation}\label{semichar}
\psi(xy)=\psi(x)\psi(y),\quad \psi(e)=1.
 \end{equation}

We denote by $\widetilde{S}$ the semigroup of all  semicharacters of $S$.

The  definition of the Laplace transform of a bounded Radon measure $\mu$ on $S$ due to N. Bourbaki   and looks as follows
 \begin{equation}\label{generalLaplace}
\widetilde{\mu}(\psi)=\int_S \psi(g)\,d\mu(g)
 \end{equation}
where  $\psi\in \widetilde{S}$ is continuous \cite[Chapter IX, \S 5]{Bourbaki}.

For the additive semigroup $S=\mathbb{R}_+^n$ \eqref{generalLaplace} turns into the classical Laplace transform. On the other hand, if $S$ is a locally compact  Abelian group \eqref{generalLaplace} turns into the Fourier transform on $S$.

But it is  easy to verify that, e.~g.,  the compact Abelian topological semigroup of idempotents  $\mathbb{I}_{\min}:=([0, 1], \min)$ with the semigroup operation $x\cdot y=\min(x,y)$ has the unique continuous semicharacter, namely, $\psi\equiv 1$.

This motivates the following more general definition.

\begin{definition}\label{df:LaplaseTrans}
We define the Laplace transform of a given Borel measure $\mu$ on a locally compact Abelian topological semigroup  $S$ by the formula \eqref{generalLaplace} where $\psi$  runs over such Borel measurable semicharactesrs of $S$ that  the wright-hand side exists.
 \end{definition} 

In the following by a semilattice we mean as usual  a commutative semigroup of idempotents. If in addition a semilattice is a topological semigroup, we call it  a topological semilattice (see, e.~g., \cite{Carruth}).

Professor Yu. I. Lyubich noted that the general form of inclusion and exclusion principle   (see \cite[p. 1049]{Graham}) can be interpreted as the converse of the Laplace transform of a function on a  semilattice $2^M$ of all subsets of a given finite set $M$ with set-theoretical  union as a semigroup operation \cite{Lyubich} and
 posed the problem of inverting the Laplace transform of  measures on  more general  topological semilattices  (oral communication). The purpose of this note is to provide partial solutions to this problem. Specifically, Theorems \ref{th:Laplace Inverse} and \ref{th:Laplace Inverse_Top} answer this question for discrete semilattices and for the class of topological semilattices containing subsemilattices of Boolean algebras. Since the resulting formulas differ significantly, we assume that the answer to Professor Yu. I. Lyubich's question depends significantly on the type of semilattices under consideration (similar to the fact that the inverse of the classical Laplace transform of a measure on $\mathbb{R}_+^n$ and the inverse of the classical Fourier transform of a measure on $\mathbb{R}^n$ differ significantly).

 \section{Laplace transform on  semilattices, the discrete case}

\subsection{Semicharacters of discrete semilattices}

According to  \cite[Theorem 2.7]{HZ}  any  finite or infinite, semilattice $G$ has a faithful  representation as a system of subsets of $G$ with the union as a semigroup operation. 

Let $\mathcal{S}=(\mathcal{S},\cup)$ be  a semilattice of subsets of a given set $M$ with the set-theoretical union as a semigroup operation, $\varnothing\in \mathcal{S}$. 

A semicharacter of $\mathcal{S}$ is any function  on $\mathcal{S}$ that satisfies the relations
$$
\psi(A\cup B)=\psi(A)\psi(B),\quad \psi(\varnothing)=1.
$$
It is obvious that every $\psi$ that is not identically $1$ takes the values $0$ and $1$ only.

The following proposition gives a full  description  of semicharacters of  $\mathcal{S}$.
 
 \begin{prop}\label{pr:semichar}
  Let $\mathcal{S}$ be some semilattice of finite subsets of $M$ with the set-theoretical union as a semigroup operation. Then  $\psi\in\widetilde{\mathcal{S}}$ if and only if  there is such $M_{\psi}\subseteq M$ (``the support set of $\psi$'') that 
$$
\psi(A)=
\begin{cases}
1, \mbox { if } A\subseteq M_\psi\\
0, \mbox { otherwise}
\end{cases}
\ \ (A\in\mathcal{S}).
$$
 \end{prop}
 
 \begin{proof} The ``if''  part can be verified by a straightforward  calculations. 

Now let   $\psi\in\widetilde{\mathcal{S}}$ and  
$$
\mathcal{S}(\psi):=\{B\in \mathcal{S}: \psi(B)=1\}.
$$ 
The family  $\mathcal{S}(\psi)$ is a subsemilattice of  $\mathcal{S}$ and its compliment  $\mathcal{S}\setminus\mathcal{S}(\psi)$  is an ideal of  $\mathcal{S}$. We put
$$
M_\psi:=\cup\{B: B\in \mathcal{S}(\psi)\}.
$$
 Two cases are possible for a set $A\in \mathcal{S}$.
 
 (a) $A\subseteq M_\psi$. Since $A$ is finite, there are such $B_i\in \mathcal{S}(\psi)$ that $A\subseteq  B$, where $B=\cup_{i=1}^nB_i\in \mathcal{S}(\psi)$. Now   $ B=A\cup B$ implies that $\psi(A)=1$.
 
 (b) $A$ is not a subset of $M_\psi$. Then  $A\in \mathcal{S}\setminus\mathcal{S}(\psi)$ and hence $\psi(A)=0$.
 \end{proof}
 
 In other wards this proposition states that the general form of semicharacters of $\mathcal{S}$ is
 $$
\psi_X(A)=
\begin{cases}
1, \mbox { if } A\subseteq X\\
0, \mbox { otherwise}
\end{cases}
\ \ (A\in\mathcal{S})
$$
 where $X$ is any nonempty subset of $M$. It is clear that $M_{\psi_X}\subseteq X$.

 \subsection{The inverse formula for the Laplace transform on discrete semilattices}

 In  this section, as above,  we endow $\mathcal{S}$ with the discrete topology. In this case, the general definition of the Laplace transform  of a measure  $\mu$ on $\mathcal{S}$ takes the form
 $$
 \widetilde{\mu}(\psi_X)=\mu(\mathcal{S}(\psi_X))
 $$
if the right-hand side makes sense.  
  
We shall consider measures of the following form. For a given function $\Phi:\mathcal{S}\to\mathbb{R}_+$  let
$$
\frak{R}:=\{\mathcal{A}\subseteq \mathcal{S}: \mathcal{A} \mbox{ consists of finite sets  and } \sum_{A\in \mathcal{A}}\Phi(A)<\infty\}.
$$
 Then $\frak{R}$ is a ring of subsets of $\mathcal{S}$. The map on  $\frak{R}$ defined as
\begin{equation}\label{mu}
 \mu_\Phi(\mathcal{A}):=\sum_{A\in \mathcal{A}}\Phi(A)
 \end{equation}
  is a positive $\sigma$-additive finite measure on $\mathcal{S}$. We define  by $\mu_\Phi$ the Lebesgue extension of this measure to the $\sigma$-ring $\frak{S}$ generated by  $\frak{R}$, too.

  For this measure and for every finite $X\subseteq M$ the Laplace transform is as follows
  \begin{equation}\label{Hamiltonian}
  \widetilde{\mu_\Phi}(\psi_X)=\mu_\Phi(\{A\in \mathcal{S}:A\subseteq X\})=\sum_{A\in \mathcal{S}, A\subseteq X}\Phi(A).
  \end{equation}
 
\begin{remark} Let the ring  $\frak{R}$ be hereditary, i.~e., if  $X\in \frak{R}$ then all subsets of $X$ belong to  $\frak{R}$, too. In this case, formula \eqref{Hamiltonian} can be rewritten as follows
\begin{equation}\label{Laplace for Phi}
  \widetilde{\Phi}(X)=\sum_{A\subseteq X}\Phi(A) \quad (X\in \frak{R}),
  \end{equation}
 and we get a definition of a Laplace transform of a function $\Phi$.
 \end{remark}

 \begin{example}\label{ex:02} Let $M=\mathbb{Z}^n$,   $\frak{R}$ consists of all finite subsets of $\mathbb{Z}^n$, and $\Phi$ takes its values in some $C^*$-algebra $\frak{A}$. In this case, the map $X\mapsto \widetilde{\Phi}(X)$ given by formula \eqref{Laplace for Phi} is called  in quantum statistical mechanics a Hamiltonian  (see, e.~g., \cite{BrRob, Ruelle}). Result of this section can be generalized for  such $\Phi$, as well.
  \end{example}
 
 Now we are in a position to state and prove the main result of this section (below $|A|$ stands  for the number of elements of a finite set $A$).

\begin{theorem}\label{th:Laplace Inverse}
 Let $\mathcal{S}$ be any semilattice  of   subsets of some nonempty set $M$ the
semigroup operation being set-theoretical union, $\varnothing\in \mathcal{S}$ and  $\Phi:\mathcal{S}\to \mathbb{R}_+$. Let  $f(X):= \widetilde{\mu_\Phi}(\psi_X)$ for every finite $X\subset M$.  Then for every $\mathcal{A}\in \frak{R}$
 \begin{equation}\label{discretL-1}
 \mu_\Phi(\mathcal{A})=\sum_{A\in \mathcal{A}}\sum_{X\subseteq A} (-1)^{|A|+|X|}f(X).
\end {equation}
\end{theorem}

\begin{proof}  First note that for each finite $A \subseteq M$,  and $Y\subseteq A$ one has
\begin{equation}\label{sum}
\sum_{Y\subseteq X\subseteq A}(-1)^{|X|}=
\begin{cases}
0,  \mbox{ if } Y\ne A\\
(-1)^{|A|}, \mbox{ if }  Y= A,
\end{cases}
\end{equation}
because if $Y\ne A$ then
\begin{eqnarray*}
&\sum_{Y\subseteq X\subseteq A}(-1)^{|X|}\\
&=\sum_{X_1\subseteq A \setminus Y}(-1)^{|X_1|+|Y|}\\
&=(-1)^{|Y|}\sum_{k=0}^n(-1)^k{n\choose k}=0,
\end{eqnarray*}
where $n=|A\setminus Y|$, $k=|X_1|$.

Further, 
$$
  f(X)=\sum_{Y\in \mathcal{S}, Y\subseteq X}\Phi(Y).
  $$

Therefore,  we have for $A \in \mathcal{S}$ in view of \eqref{sum}
\begin{align*}
&(-1)^{|A|}\sum_{X \subseteq A} (-1)^{|X|} f(X)\\
&=(-1)^{|A|}\sum_{X \subseteq A} (-1)^{|X|}\sum_{Y\in \mathcal{S}, Y\subseteq X}\Phi(Y)\\
&=(-1)^{|A|}\sum_{(X,Y):Y\in \mathcal{S}, Y\subseteq X\subseteq A}(-1)^{|X|}\Phi(Y)\\
&=(-1)^{|A|}\sum_{Y\in \mathcal{S}, Y\subseteq A}\Phi(Y)\sum_{Y\subseteq X\subseteq A}(-1)^{|X|}\\
&=\Phi(A).
\end{align*}

It follows
\begin{align*}
 \mu_\Phi(\mathcal{A})&=\sum_{A\in \mathcal{A}}\Phi(A)\\
 &=\sum_{A\in \mathcal{A}}\sum_{X\subseteq A} (-1)^{|A|+|X|} f(X).
\end {align*}
   \end{proof}

\section{Laplace transform on  Lawson  semilattices of open-and-compact sets}

In this section,  $\mathcal{S}=(\mathcal{S},\cup)$ stands for a semilattice (with respect to  the  set-theoretical union) of open-and-compact subsets of a Hausdorff  topological space $M$.

\begin{example}\label{ex:2}
Recall that a topological space $M$ is called zero-dimensional if $M$ is a non-empty $T_1$-space and has a base $\mathcal{B}$ consisting of open-and-closed sets. 

According to the Stone representation theorem every Boolean algebra $(\mathbf{B},\cup,\cap,-)$  is isomorphic to the field of all open-and-compact subsets of a compact totally  disconnected space $T$ ``the Stone space of  $(\mathbf{B},\cup,\cap,-)$'' (see, e.~g., \cite{Sikorski}). The space $T$ being compact is zero-dimensional.  Therefore every subsemilattice $(S,\cup)$ of $(\mathbf{B},\cup,\cap,-)$ is isomorphic to some semilattice of  open-and-compact subsets of  the space $T$ with set-theoretical union as a semigroup operation.
\end{example}

\begin{example}\label{ex:1}
More generally let $M$ be compact and zero-dimensional space and the base  $\mathcal{B}$ of its topology consists of open-and-compact sets. Any subsemilattice of the  semilattice $\mathcal{K}$ of all compact subsets of $M$ with respect to the set-theoretic union generated by a subfamily of $\mathcal{B}$ consists of open-and-compact subsets of $M$. 
\end{example}

We introduce the topology in $\mathcal{S}$ as follows.  Let for an arbitrary sets $E, G_1, \dots, G_n$$\subseteq M$ let
\begin{align*}
\mathcal{V}(E; G_1, \dots, G_n)\\
&:=\{A\in \mathcal{S}:  A\cap E=\varnothing,  A\cap G_i\ne \varnothing, i=1,\dots,n\}.
\end{align*}
The open base $\frak{B}$ of the topology $\frak{T}$ in  $\mathcal{S}$ consists  of sets $\mathcal{V}(F; U_1, \dots, U_n)$ 
where $F$ runs over closed subsets of $M$, and all $U_i$ run over open-and-closed subsets of $M$. 

Recall that a topological semilattice $S$ is said to be
Lawson if it has a neighborhood basis of subsemilattices.
Equivalently, it is Lawson if the semilattice homomorphisms into
$\mathbb{I}_{\min}=([0, 1], \min)$ separate the points of $S$ \cite{Lawson1}  (see also\cite{Carruth}).

\begin{lemma}\label{Top_Sem} (i) The family $\frak{B}$ form a base of a Hausdorff  topology $\frak{T}$ on  $\mathcal{S}$.  

(ii) The space  $\mathcal{S}$ endowed with the topology $\frak{T}$ is  an
 Hausdorff  topological Lawson semilattice with respect to  the  set-theoretical union as the semigroup operation.  
\end{lemma}

\begin{proof} (i) Since $A\in \mathcal{V}(M\setminus A, M)$ for each $A\in \mathcal{S}$ and
\begin{align}\label{cap}
\mathcal{V}(F; U_1, \dots, U_n)&\cap \mathcal{V}(F'; U_1', \dots, U_n')\\ \nonumber
&=\mathcal{V}(F\cup F'; U_1, \dots, U_n,U_1', \dots, U_n'),
\end{align}
$\frak{B}$ is a base of a  topology on $\mathcal{S}$. This topology is Hausdorff, because for $A,B\in \mathcal{S}$ such that, say, $A\setminus B\ne \varnothing$ one has
$$
A\in \mathcal{V}(M\setminus A;A\setminus B), B\in \mathcal{V}(A\setminus B, B).
$$

(ii) To prove that the map  $(A,B)\mapsto A\cup B$ is continuous,  let $A\cup B\in \mathcal{V}(F; U_1, \dots, U_n)$, in other wards, $(A\cup B)\cap F =\varnothing$ and $(A\cup B)\cap U_k \ne \varnothing$ for all $k$. Let $I:=\{i:A\cap U_i\ne \varnothing\}$,  $J:=\{j:B\cap U_j\ne \varnothing\}$.  We put  $\mathcal{V}_1:=\mathcal{V}(F; U_i, i\in I)$ if $I\ne \varnothing$, and $\mathcal{V}_1:=\mathcal{V}(F;M)$ if $I= \varnothing$,  $\mathcal{V}_2:=\mathcal{V}(F; U_j, j\in J)$ if $J\ne \varnothing$, $\mathcal{V}_2:=\mathcal{V}(F;M)$ if $J= \varnothing$. Then  $\mathcal{V}_1$ is a neighbourhood of $A$, $\mathcal{V}_2$ is a neighbourhood of $B$.

If $A'\in \mathcal{V}_1$ and  $B'\in \mathcal{V}_2$ we have $A'\cup B'\in \mathcal{V}(F; U_1, \dots, U_n)$ which proves the continuity.

Since every family $\mathcal{V}(F; U_1, \dots, U_n)$ is  a semilattice  with respect to  the  set-theoretical union as the semigroup operation,  $\mathcal{S}$ is Lowson.
\end{proof}

\begin{remark}\label{locallyC} The so-called \it{myope topology} on the family $\mathcal{K}$  of all compact subsets of a locally compact space $M$ is defined by a base of sets $\mathcal{V}(F; U_1, \dots, U_n)$  where $F$ runs over closed subsets of $M$, and all $U_i$ run over open subsets of $M$ \cite[p. 268]{Berg}. It follows that our topology $\frak{T}$ is weaker that the topology induced on $\mathcal{S}$ by the myope topology. It is known that  $\mathcal{K}$ is locally compact in myope topology \cite[p. 269, remark 2.12]{Berg}. Then the space $(\mathcal{S}, \frak{T})$ is locally compact if $\mathcal{S}$ is an intersection of an open and closed subsets of $\mathcal{K}$.
\end{remark}

\subsection{Semicharacters of  semilattices of open-and-compact sets}

Recall that we assume that the semilattice $\mathcal{S}$ consists of open-and-compact subsets of a given  Hausdorff topological space $M$, and we write $\widetilde{\mathcal{S}}$ for the semigroup of all semicharacters of  $\mathcal{S}$.

\begin{prop}\label{pr:semichar_top} Every  semicharacter of  $\mathcal{S}$ is Borel measurable and has  exactly the form
\begin{align}\label{semicharacterTop}
\psi_U(A)=
\begin{cases}
1, \mbox { if } A\subseteq U\\
0, \mbox { otherwise}
\end{cases}
\ \ (A\in\mathcal{S})
\end{align}
 where $U$ is any nonempty open subset of $M$. 
\end{prop}

\begin{proof} Let $\psi$ be  a  semicharacter of  $\mathcal{S}$.  As in the proof of Proposition \ref{pr:semichar}  let $\mathcal{S}(\psi):=\{A\in \mathcal{S}: \psi(A)=1\}$ and consider the open set $U:=\cup\{A:A\in \mathcal{S}(\psi)\}$.  Two cases are possible for a set $K\in \mathcal{S}$.
 
 (a) $K\subseteq U$. Since $K$ is compact, there are such open  $B_i\in \mathcal{S}(\psi)$ that $K\subseteq  B$, where $B=\cup_{i=1}^n B_i\in \mathcal{S}(\psi)$. Now   $ B=K\cup B$ implies that $\psi(K)=1$.
 
 (b) $K$ is not a subset of $U$. Then  $K\in \mathcal{S}\setminus\mathcal{S}(\psi)$ and hence $\psi(K)=0$.

 Then $\psi=\psi_U$.
 
 Further, since both sets
 $$
 \{A\in \mathcal{S}: \psi(A)=1\}=\mathcal{V}(M\setminus U,M)
 $$
 and
 $$
 \{A\in \mathcal{S}: \psi(A)=0\}=\mathcal{S}\setminus\{A\in \mathcal{S}: \psi(A)=1\}
 $$
 are Borel in $\mathcal{S}$, $\psi$ is Borel measurable.

The  statement that  $\psi_U\in \widetilde{\mathcal{S}}$ can be verified  by a straightforward  calculations.
\end{proof}

 \subsection{The inverse formula for the Laplace transform on  semilattices of open-and-compact sets.}
 
 In the following we assume  that  $\mu$ is a  regular Borel measure on $\mathcal{S}$.

\begin{theorem}\label{th:Laplace Inverse_Top}
 Let $\mathcal{S}$ be as above. Let  $f(U):= \widetilde{\mu}(\psi_U)$  where $U$ is any nonempty open subset of $M$.   Then for every set $\mathcal{V}(F; U_1, \dots, U_n)$ in  the base $\frak{B}$ one has
 \begin{equation}\label{Top_L-1}
 \mu(\mathcal{V}(F; U_1, \dots, U_n))=(-1)^n(\Delta_{U_1}\circ\dots\circ\Delta_{U_n} f')(F),
\end {equation}
where $f'(F):=f(M\setminus F)$, and  $(\Delta_{U}\phi)(A)= \phi(A\cup U)-\phi(A)$  is a difference operator.

Moreover, the measure $\mu$ is uniquely determined by its values on $\frak{B}$ given by \eqref{Top_L-1}.
\end{theorem}

\begin{proof} Let
$$
\mathcal{S}^K:=\{A\in \mathcal{S}: A\cap K=\varnothing\}.
$$
Then 
$$
\mathcal{S}^F\setminus \cup_{i=1}^n\mathcal{S}^{U_i}=\mathcal{V}(F; U_1, \dots, U_n).
$$
We prove \eqref{Top_L-1} by induction. 
Since 
\begin{equation*}
f(M\setminus F)= \int_{\mathcal{S}}\psi_{M\setminus F}(A)d\mu(A)=\mu(\mathcal{S}(\psi_{M\setminus F}))=\mu(\mathcal{S}^F),
\end {equation*}
for $n=1$ we have
\begin{align*}
-(\Delta_{U_1}f')(F)&=f(M\setminus F)-f(M \setminus (F\cup U_1))\\
&=\mu(\mathcal{S}^F)-\mu(\mathcal{S}^{F\cup U_1})\\
&=\mu(\mathcal{S}^F\setminus\mathcal{S}^{F\cup U_1})\\
&=\mu(\mathcal{V}(F; U_1)).
\end{align*}
Now  assume that \eqref{Top_L-1} holds for some $n$. Then
\begin{align*}
&(-1)^{n+1}(\Delta_{U_1}\circ\dots\circ\Delta_{U_{n+1}} f')(F)\\
&=(-1)^{n+1}\Delta_{U_1}\circ(\Delta_{U_2}\circ\dots\circ\Delta_{U_{n+1}} f')(F)\\
&=(-1)^{n+1}((\Delta_{U_2}\circ\dots\circ\Delta_{U_{n+1}} f')(F\cup U_1) - (\Delta_{U_2}\dots\circ\Delta_{U_{n+1}} f')(F))\\
&=(-1)^{n}(\Delta_{U_2}\dots\circ\Delta_{U_{n+1}} f')(F)-(-1)^{n}(\Delta_{U_2}\dots\circ\Delta_{U_{n+1}} f')(F\cup U_1)\\
&=\mu((\mathcal{S}^F\setminus \cup_{i=2}^{n+1}\mathcal{S}^{U_i})-\mu(\mathcal{S}^{F\cup U_1}\setminus \cup_{i=2}^{n+1}\mathcal{S}^{U_i})\\
&=\mu((\mathcal{S}^F\setminus \cup_{i=2}^{n+1}\mathcal{S}^{U_i})\setminus (\mathcal{S}^{F\cup U_1}\setminus \cup_{i=2}^{n+1}\mathcal{S}^{U_i}))\\
&=\mu(\mathcal{V}(F; U_1,U_2,\dots,U_{n+1}))
\end{align*}
(above we used the equality  $(A\setminus C)\setminus (B\setminus C)= A\setminus(B\cup C)$).
This proves   the first stement.

Further, let    regular Borel measures $\mu_1$ and $\mu_2$   on $\mathcal{S}$ coincide on $\frak{B}$. Since,  by \eqref{cap}, $\frak{B}$ is closed with respect to finite intersections, the inclusion-and-exclusion rule implies that   $\mu_1$ and $\mu_2$ coincide on  finite unions of elements of $\frak{B}$. If we denote by  $\frak{T}_f$ the family of such finite unions, then for every compact $\mathcal{K}\subseteq \mathcal{S}$ we have by the outer regularity of $\mu_i$
\begin{align*}
\mu_i(\mathcal{K})&=\inf\{\mu_i(\mathcal{U}): \mathcal{K} \subseteq \mathcal{U}, \mathcal{U} \in \frak{T}\}\\
&=\inf\{\mu_i(\mathcal{W}): \mathcal{W}  \in  \frak{T}_f\}\ \ (i=1,2).
\end{align*}
Thus (inner regular) measures $\mu_1$ and $\mu_2$ coincide on compact subsets of $\mathcal{S}$, as well.

This completes the proof.
\end{proof}

\end{document}